\documentclass[a4paper,11pt]{article}
\usepackage{amsfonts}
\usepackage[T1]{fontenc}
\usepackage{microtype}
\usepackage{makeidx}
\usepackage{setspace}
\usepackage{authblk}
\usepackage{graphicx}
\usepackage{xcolor}
\usepackage[all]{xy}
\usepackage{amsmath}
\usepackage{amssymb}
\usepackage{booktabs}
\usepackage{braket}
\usepackage{array}
\usepackage{tabularx}
\usepackage{multirow}
\usepackage{indentfirst}
\usepackage{cite}
\usepackage{stmaryrd}
\usepackage{faktor}
\usepackage{titling}
\usepackage{tikz}
\usepackage{dsfont}
\usepackage{url}
\usepackage{hyperref}

\usepackage{tkz-graph}

\usetikzlibrary{matrix,arrows,decorations.pathmorphing}
    \renewcommand{\leq}{\leqslant}
    
\usepackage{amsthm}
\theoremstyle{plain}
\newtheorem{thm}{Theorem}[section]
\newtheorem*{thm*}{Theorem}

\newtheorem{conge}[thm]{Conjecture}

\theoremstyle{remark}
\newtheorem{oss}[thm]{Remark}

\theoremstyle{definition}

\title{\bf{Equivalence between invariance conjectures for parabolic Kazhdan-Lusztig polynomials}}
\author{}
\author{Paolo Sentinelli\thanks{ Dipartimento di Matematica, Politecnico di Milano, Milan, Italy. \\ \href{mailto:paolosentinelli@gmail.com}{paolosentinelli@gmail.com}}}
\date{}
\begin{document}

\maketitle

\vspace{-4em}

\begin{abstract}
We prove that the combinatorial invariance conjecture for parabolic Kazhdan-Lusztig polynomials, formulated by Mario Marietti, is equivalent to its restriction to maximal quotients. This equivalence lies at the other extreme in respect to the equivalence, recently proved by Barkley and Gaetz, with the invariance conjecture for Kazhdan-Lusztig polynomials, which turns out to be equivalent to the conjecture for maximal quotients.
\end{abstract}

$\,$

\section{Introduction}
\vspace{1em}
The parabolic Kazhdan-Lusztig polynomials have been introduced by V. Deodhar in \cite{deodhar}. They define canonical bases of modules of Hecke algebras, induced by trivial representations, the ones of type $x=q$, and alternating representations, the ones of type $x=-1$; in this second case, the parabolic Kazhdan-Lusztig polynomials are related to the intersection cohomology of Schubert varieties in $G/P$, for Kac-Moody
groups $G$ and standard parabolic subgroups $P$ (see \cite[Sec. 4]{deodhar} and \cite{kashiwara}). For type $x=q$, a categorification has been realized by Libedinsky and Williamson \cite{libedinsky}; this provides an interpretation of the parabolic Kazhdan-Lusztig polynomials of type $q$ as graded ranks of morphism spaces, and hence they are non--negative polynomials.

M. Marietti \cite[Conjecture 1.3]{marietti1} has proposed, and proved for lower Bruhat intervals in \cite{marietti2}, the following invariance conjecture for parabolic Kazhdan-Lusztig polynomials:
\vspace{1em}
\begin{conge} \label{cong 2} [Marietti]
Let $(W,S)$, $(W',S')$ be Coxeter systems, $J\subseteq S$ and $J'\subseteq S'$. If there exists a poset isomorphism $[u,v]^J \simeq [u',v']^{J'}$ which extends to an isomorphism $[u,v]\simeq [u',v']$, then $P^{J,x}_{u,v}=P^{J',x}_{u',v'}$, for all $x\in \{-1,q\}$.    
\end{conge}
\vspace{1em}
Conjecture \ref{cong 2} restricts, for $J=\varnothing$, to the well known invariance conjecture of Kazhdan-Lusztig polynomials.

\vspace{1em}
\begin{conge}\label{cong 1} [Lusztig, Dyer] Let $[u,v]$ and $[u',v']$ be Bruhat intervals in the Coxeter systems $(W,S)$ and $(W',S')$, respectively. Then 
$$[u,v] \simeq [u',v'] \, \, \Rightarrow \, \, P_{u,v}=P_{u',v'}.$$ 
\end{conge}

\vspace{1em}
It is known, by a result of Deodhar, that the Bruhat order splits along maximal quotients (see \cite[Theorem 2.6.1]{BB}). An analogous result for Kazhdan-Lusztig polynomials is not known; nevertheless, globally, i.e. by considering the whole class of Coxeter systems, the knowledge of parabolic Kazhdan-Luszig polynomials for maximal quotients is enough to compute all other polynomials of any quotient. This is the content of \cite[Theorem 5.8]{io} (see Theorem \ref{cor massimale} below). From this fact one can argue that also Conjecture \ref{cong 2} can be restricted to maximal quotients; this is in fact true and it is proved in Theorem \ref{teorema equiv}.  

In \cite{gaetz}, Barkley and Gaetz have recently provided an equivalence between Conjectures \ref{cong 2} and \ref{cong 1}; therefore, our equivalence between Conjectures \ref{cong 2} and \ref{cong 3} implies an equivalence between two conjectures that lie at the extreme sides: the side of “minimal” quotients and the one of maximal quotients. 
Then, by resuming the previous assertions, we have the following (Theorem \ref{teorema equiv}): 

$$\mbox{Conjecture \ref{cong 1} \, $\Leftrightarrow$ \,  Conjecture \ref{cong 2} \, $\Leftrightarrow$ \,  Conjecture \ref{cong 3}}.$$

 \vspace{1em}

We observe that the equivalence between Conjectures \ref{cong 2} and \ref{cong 3} could be not true if we restrict to a specific set of Coxeter systems. This means that partial results on Conjecture \ref{cong 3} do not necessarily correspond to partial results on Conjecture \ref{cong 2}. Analogous considerations, but for different reasons, can be done for the equivalence between Conjectures \ref{cong 1} and \ref{cong 2} (see \cite{gaetz}).

\vspace{1em}
\section{Equivalences of conjectures}
 \vspace{1em}
In this section we formulate a combinatorial invariance conjecture  for parabolic Kazhdan-Lusztig polynomials by restricting Conjecture \ref{cong 2} to maximal quotients. We then prove that these two conjectures are equivalent.

We refer to \cite{io} for notation and preliminaries. The only difference to be pointed out is the side of the indices in $W^J$ and ${^{J}W}$. Here we let, for a Coxeter system $(W,S)$ and $J\subseteq S$, 
\begin{gather*} W^J:=\Set{w\in W:\ell(w)<\ell(ws)~\forall~s\in J},
\\ {^JW}:=\Set{w\in W:\ell(w)<\ell(sw)~\forall~s\in J}.
\end{gather*} On the contrary, in \cite{io} the upper indices are flipped, so some results have to be considered accordingly to this switch. 
Given a Coxeter system, we consider the Bruhat order on $W$ and, on any subset of $W$, the induced order. For $J\subseteq S$,  we let $$[u,v]^J:=\{z\in W^J: u \leq z \leq v\} \, \subseteq \, [u,v] := \{z\in W: u \leq z \leq v\},$$ for all $u,v\in W^J$ such that $u\leq v$.
Let us formulate the restriction of Conjecture \ref{cong 2} to maximal quotients.

\vspace{1em}
\begin{conge} \label{cong 3} Let $(W,S)$, $(W',S')$ be Coxeter systems, $s\in S$ and  $s'\in S'$. If there exists a poset isomorphism $[u,v]^{S\setminus \{s\}} \simeq [u',v']^{S\setminus \{s'\}}$ which extends to an isomorphism $[u,v]\simeq [u',v']$, then $P^{S\setminus \{s\},x}_{u,v}=P^{S'\setminus \{s'\},x}_{u',v'}$, for all $x\in \{-1,q\}$.  
\end{conge}

\vspace{1em}
One of the results of \cite{io} is that any parabolic Kazhdan-Lusztig polynomial of a Coxeter system $(W,S)$ is equal to a parabolic Kazhdan-Lusztig polynomial of a maximal quotient of some Coxeter system $(\tilde{W},\tilde{S})$, as stated in the next theorem (see \cite[Theorem 5.8]{io}), 
where we denote by $P^{J,x}_{u,v}[W,S]$ a polynomial computed in the Coxeter system $(W,S)$.

\vspace{1em}
\begin{thm} \label{cor massimale} Let $(W,S)$ be a Coxeter system, $J\subseteq S$ and $v,w\in W^J$. Let $\tilde{S}:=S\cup \{\tilde{s}\}$, with $\tilde{s}\not \in S$, and $(\tilde{W},\tilde {S})$ be a Coxeter system whose Coxeter matrix $m'$ satisfies: 
\begin{enumerate}
    \item $m'(s,t)=m(s,t)$, for all $s,t\in S$;
    \item $m'(\tilde{s},s)=2$, for all $s\in J$;
    \item $m'(\tilde{s},s)\neq 2$, for all $s\in S\setminus J$,
\end{enumerate} where $m$ is the Coxeter matrix of $(W,S)$.
Then 
    $P^{J,x}_{v,w}[W,S]=P^{S,x}_{v\tilde{s},w\tilde{s}}[\tilde{W},\tilde{S}]$, for all $x\in \{-1,q\}$.
\end{thm}

\vspace{1em}
For example, if $(W,S)$ is of  type $A_n$, with $S=\{s_1,\ldots,s_n\}$, and $J=S\setminus \{s_1,s_n\}$, we have that $(\tilde{W},\tilde{S})$ can be chosen of type $\tilde{A}_n$. Hence we have the equality $P^{S\setminus \{s_1,s_n\},x}_{v,w}[A_n]=P^{S,x}_{vs_{n+1},ws_{n+1}}[\tilde{A}_n]$, for all $v,w \in W^{S\setminus \{s_1,s_n\}}$, where the set of generators of $\tilde{A}_n$ is $S\cup\{s_{n+1}\}$.

 \vspace{1em}
We now can prove our announced result.
\begin{thm} \label{teorema equiv} We have the following equivalences of conjectures:
    $$\mbox{Conjecture \ref{cong 1} \, $\Leftrightarrow$ \,  Conjecture \ref{cong 2} \, $\Leftrightarrow$ \,  Conjecture \ref{cong 3}}$$
\end{thm}
\begin{proof}
    The first equivalence is part of \cite[Theorem 1.6]{gaetz}. We prove now the second one. It is obvious that Conjecture \ref{cong 2} implies Conjecture \ref{cong 3}. So assume Conjecture \ref{cong 3}. Let $J\subseteq S$, $J'\subseteq S'$ and $\phi: [u,v]^{J}\rightarrow [u',v']^{J'}$ be a poset isomorphism which extends to an isomorphism $\Phi: [u,v]\rightarrow [u',v']$. Let $(\tilde{W},\tilde{S})$ and $(\tilde{W}',\tilde{S}')$ be as in Theorem \ref{cor massimale}. The function $\tilde{\Phi}:[u\tilde{s},v\tilde{s}]\rightarrow [u'\tilde{s}',v'\tilde{s}']$ defined by $\tilde{\Phi}(z\tilde{s})=\Phi(z)\tilde{s}'$, for all $z\in [u,v]$, is a poset isomorphism. 

    By specializing the proof of \cite[Proposition 3.2]{io} to $W=\tilde{W}$, $I=J=S$ and $u=\tilde{s}$, and noting that by construction $\tilde{W}_S=W$,  we have a poset isomorphism
    $f: W^J \rightarrow  \tilde{W}^S\cap W\tilde{s}$ such that $f(x)=x\tilde{s}$, for all $x \in W^J$.
    Hence we have that
    $[u\tilde{s},v\tilde{s}]^S=\{z\tilde{s}: z\in [u,v]^J\}$. Similarly 
    $[u'\tilde{s}',v'\tilde{s}']^{S'}=\{z\tilde{s}': z\in [u',v']^{J'}\}$.
    This implies $\tilde{\Phi}(z\tilde{s})=\phi(z)\tilde{s}' \in [u'\tilde{s}',v'\tilde{s}']^{S'}$, for all $z\in [u,v]^J$, and then 
    the restriction $\tilde{\Phi}|_{[u\tilde{s},v\tilde{s}]^S}$ provides an isomorphism $[u\tilde{s},v\tilde{s}]^S\simeq  [u'\tilde{s}',v'\tilde{s}']^{S'}$. Therefore, by Theorem \ref{cor massimale} and having assumed Conjecture \ref{cong 3}, we obtain
    $$P^{J,x}_{u,v}[W,S]=P^{S,x}_{u\tilde{s},v\tilde{s}}[\tilde{W},\tilde{S}]=P^{\tilde{S}',x}_{u'\tilde{s}',v'\tilde{s}'}[\tilde{W}',\tilde{S}']=P^{J',x}_{u',v'}[W',S'],$$ for all $x\in \{-1,q\}$.
\end{proof}

 \vspace{1em}
We end this note by the following observation.
Let $X\subseteq \{\infty, 3,4,5,\ldots\}$ be such that $X\neq \varnothing$.
We say that a Coxeter system $(W,S)$ with Coxeter matrix $m$ is \emph{of class} $X$ if $m(s,t) \in X\cup \{1,2\}$.
For example, for $X=\{3\}$ we obtain the class of simply-laced Coxeter systems and, for $X=\{\infty\}$, the one of right-angled Coxeter systems. If $(W,S)$ is of class $X$, then $(\tilde{W},\tilde{S})$ can be constructed of class $X$. Hence, if we restrict Conjectures  \ref{cong 2} and \ref{cong 3} to any class $X$, the second equivalence stated in Theorem \ref{teorema equiv} is still true. Since also the first equivalence remains true when the conjectures are restricted to any class, we conclude that the three equivalences of Theorem \ref{teorema equiv} are still true when the conjectures are restricted to any class $X$.

\begin{oss}
    All results stated in this paper for the parabolic Kazhdan-Lusztig polynomials can be formulated also for the parabolic $R$-polynomials. 
\end{oss}

\section{Acknowledgements}
I would like to thanks Grant Barkley, Christian Gaetz and Mario Marietti for some useful conversations, which have been possible as participant of the Workshop “Bruhat order: recent developments and open problems”, held
at the University of Bologna.
I also thanks Davide Bolognini for pointing me out the result of Barkley and Gaetz.

\end{document}